\documentclass[12pt]{article}
\usepackage{lineno,hyperref}
\modulolinenumbers[5]
\usepackage{color}

\usepackage{graphics,amsmath,amssymb}
\usepackage{amsthm}
\usepackage{amsfonts}
\usepackage{latexsym}
\usepackage{epsf}

\newcommand{\stirlings}[2]{\genfrac\{\}{0pt}{}{#1}{#2}}
\newcommand{\stirlingf}[2]{\genfrac[]{0pt}{}{#1}{#2}}

\theoremstyle{plain}
\newtheorem{theorem}{Theorem}
\newtheorem{corollary}[theorem]{Corollary}

\theoremstyle{definition}

\theoremstyle{remark}

\begin{document}

\begin{center}
\vskip 1cm{\LARGE\bf The Translated Whitney-Lah Numbers: Generalizations and $q$-Analogues}
\vskip 1cm
\large    
Mahid M. Mangontarum\\
Department of Mathematics\\
Mindanao State University-Main Campus\\
Marawi City 9700\\
Philippines \\
\href{mailto:mmangontarum@yahoo.com}{\tt mmangontarum@yahoo.com} \\
\href{mailto:mangontarum.mahid@msumain.edu.ph}{\tt mangontarum.mahid@msumain.edu.ph}\end{center}

\vskip .2 in

\begin{abstract}
In this paper, we derive formulas for the translated Whitney-Lah numbers and show that they are generalizations of already-existing identities of the classical Lah numbers. $q$-analogues of the said formulas are also obtained for the case of the translated $q$-Whitney numbers.
\end{abstract}

\section{Introduction}

The (unsigned) Lah numbers, denoted by $L(n,k)$, count the number of partitions of a set $X$ with $n$ elements into $k$ nonempty linearly ordered subsets. These numbers are known to satisfy the following basic combinatorial properties:
\begin{itemize}
\item explicit formula
\begin{equation}
L(n,k)=\frac{n!}{k!}\binom{n-1}{k-1};\label{lah1}
\end{equation}
\item recurrence relation
\begin{equation}
L(n+1,k)=L(n,k-1)+(n+k)L(n,k);\label{lah2}
\end{equation}
\item exponential generating function
\begin{equation}
\sum_{n=0}^{\infty}L(n,k)\frac{t^n}{n!}=\frac{1}{k!}\left(\frac{t}{1-t}\right)^k.\label{lah3}
\end{equation}
\end{itemize}
The numbers $L(n,k)$ are also known to be coefficients of rising factorials in terms of falling factorials. That is,
\begin{equation}
\left\langle t\right\rangle_n=\sum_{k=0}^nL(n,k)(t)_k,\label{LahDef}
\end{equation}
where $$\left\langle t\right\rangle_n=t(t+1)(t+2)\cdots(t+n-1)$$
is the rising factorial of $t$ of order $n$ and $$(t)_k=t(t-1)(t-2)\cdots(t-k+1)$$
is the falling factorial of $t$ of order $k$ with $\left\langle t\right\rangle_0=(t)_0=1$ and $(-t)_n=(-1)^n\left\langle t\right\rangle_n$. The Lah numbers $L(n,k)$ are actually closely related with the well-known Stirling numbers. To illustrate, we first recall that the Stirling numbers of the first and second kinds, denoted by $\stirlingf{n}{k}$ and $\stirlings{n}{k}$, respectively, are defined as coefficients in the expansions of the relations
\begin{equation}
(t)_n=\sum_{k=0}^n(-1)^{n-k}\stirlingf{n}{k}t^k\label{sn1}
\end{equation}
and 
\begin{equation}
t^n=\sum_{k=0}^n\stirlings{n}{k}(t)_k.\label{sn2}
\end{equation}
Notice that putting $-t$ in place of $t$ in \eqref{sn1} yields
\begin{equation}
\left\langle t\right\rangle_n=\sum_{k=0}^n\stirlingf{n}{k}t^k.\label{sn1.1}
\end{equation}
By substituting \eqref{sn2} in the right-hand side of \eqref{sn1.1},
\begin{eqnarray*}
\left\langle t\right\rangle_n&=&\sum_{k=0}^n\stirlingf{n}{k}\sum_{j=0}^k\stirlings{k}{j}(t)_j\\
&=&\sum_{j=0}^n\left(\sum_{k=j}^n\stirlingf{n}{k}\stirlings{k}{j}\right)(t)_j.
\end{eqnarray*}
Combining this with \eqref{LahDef} and by comparing the coefficients of $(t)_j$, we are able to write
\begin{equation}
L(n,k)=\sum_{j=k}^n\stirlingf{n}{j}\stirlings{j}{k}.\label{lahsS}
\end{equation}
It is important to note that here, the numbers $\stirlingf{n}{k}$ particularly refer to the ``unsigned'' Stirling numbers of the first kind which count the number of permutations of the $n$-element set $X$ into $k$ disjoint cycles. Similarly, the Stirling numbers of the second kind $\stirlings{n}{k}$ can be combinatorially interpreted as the number of partitions of $X$ into $k$ nonempty blocks. With this, the Bell numbers $B_n$ are defined as the total number of partitions of the $n$-element set $X$. That is,
\begin{equation}
B_n=\sum_{k=0}^n\stirlings{n}{k}.
\end{equation}
The paper of Petkov\v{s}ek and Pisanski \cite{Pet}, and the books of Comtet \cite{Comt} and Chen and Kho \cite{Chen} contain detailed discussions on the Lah, Stirling and Bell numbers especially their respective combinatorial properties and interpretations. In addition to these, Qi \cite{Qi} recently obtained an explicit formula for the Bell numbers expressed in terms of both the Lah numbers and the Stirling numbers of the second kind, viz.
\begin{equation}
B_n=\sum_{k=1}^n(-1)^{n-k}\left(\sum_{\ell=1}^kL(k,\ell)\right)\stirlings{n}{k}.\label{QiF}
\end{equation}

The results of this paper are organized as follows. In Section \ref{sec1}, we present the translated Whitney numbers and derive some formulas which generalize already-existing identities for the classical Lah numbers, including one that will generalize \eqref{QiF}. In Section \ref{sec2}, we establish the $q$-analogues of some of the results in Section \ref{sec1} using as framework the translated $q$-Whitney numbers.

\section{Translated Whitney numbers}\label{sec1}

In 2013, Belbachir and Bousbaa \cite{Bel} introduced the translated Whitney numbers using a combinatorial approach which involves ``mutations'' of some elements. To be more precise, the translated Whitney numbers of first kind, denoted by $\widetilde{w}_{(\alpha)}(n,k)$, were defined as the number of permutations of $n$ elements with $k$ cycles such that the elements of each cycle can mutate in $\alpha$ ways, except the dominant one while the translated Whitney numbers of the second kind, denoted by $\widetilde{W}_{(\alpha)}(n,k)$, were defined as the number of partitions of the an $n$-element set into $k$ subsets such that the elements of each subset can mutate in $\alpha$ ways, except the dominant one. These numbers were shown to satisfy the recurrence relations \cite[Theorems 2 and 8]{Bel}
\begin{equation}
\widetilde{w}_{(\alpha)}(n,k)=\widetilde{w}_{(\alpha)}(n-1,k-1)+\alpha(n-1)\widetilde{w}_{(\alpha)}(n-1,k)
\end{equation}
and 
\begin{equation}
\widetilde{W}_{(\alpha)}(n,k)=\widetilde{W}_{(\alpha)}(n-1,k-1)+\alpha k\widetilde{W}_{(\alpha)}(n-1,k),
\end{equation}
and the horizontal generating functions \cite[Theorems 4 and 10]{Bel}
\begin{equation}
(t|-\alpha)_n=\sum_{k=0}^n\widetilde{w}_{(\alpha)}(n,k)x^k\label{wHGF1}
\end{equation}
and
\begin{equation}
x^n=\sum_{k=0}^n\widetilde{W}_{(\alpha)}(n,k)(t|\alpha)_k.\label{wHGF2}
\end{equation}
Here, we used $(t|\alpha)_n$ to denote the generalized factorial of $t$ of increment $\alpha$ defined by
\begin{equation*}
(t|\alpha)_n=\prod_{i=0}^{n-1}(t-i\alpha),\ (t|\alpha)_0=1.
\end{equation*}
In the same paper, Belbachir and Bousbaa \cite{Bel} also defined translated Whitney-Lah numbers, denoted by $\widehat{w}_{(\alpha)}(n,k)$, as the number of ways to distribute the set $\{1,2,\ldots,n\}$ into $k$ ordered lists such that the elements of each list can mutate with $\alpha$ ways, except the dominant one. The values of the numbers $\widehat{w}_{(\alpha)}(n,k)$ can be computed using the recurrence relation \cite[Theorem 13]{Bel}
\begin{equation}
\widehat{w}_{(\alpha)}(n,k)=\widehat{w}_{(\alpha)}(n-1,k-1)+\alpha(n+k-1)\widehat{w}_{(\alpha)}(n-1,k)\label{recwl}
\end{equation}
and are generated using \cite[Corollary 15]{Bel}
\begin{equation}
(t|-\alpha)_n=\sum_{k=0}^n\widehat{w}_{(\alpha)}(n,k)(t|\alpha)_k.\label{twlHGF}
\end{equation}
Similar to what is seen in \eqref{lahsS}, the translated Whitney-Lah numbers may be expressed as sum of products of $\widetilde{w}_{(\alpha)}(n,k)$ and $\widetilde{W}_{(\alpha)}(n,k)$ as follows \cite[Corollary 14]{Bel}
\begin{equation}
\widehat{w}_{(\alpha)}(n,j)=\sum_{k=j}^n\widetilde{w}_{(\alpha)}(n,k)\widetilde{W}_{(\alpha)}(k,j).\label{wlahwW}
\end{equation}
It appears that the translated Whitney and Whitney-Lah numbers are generalizations of the Stirling and Lah numbers, respectively. This may be verified by simply setting $\alpha=1$ in the defining relations of the former.

Recently, Mansour et al. \cite{Mansour} defined the recurrence relation
\begin{equation}
u(n,k)=u(n-1,k-1)+(a_{n-1}+b_k)u(n-1,k)
\end{equation}
for two sequences $(a_i)_{i\geq 0}$ and $(b_i)_{i\geq 0}$ with boundary conditions given by $u(n,0)=\prod_{i=0}^{n-1}(a_i+b_0)$ and $u(0,k)=\delta_{0,k}$, where
\begin{equation*}
\delta_{i,j}=\left\{              
	\begin{array}{cll}                   
0,\ if \ i\neq j \\   
1,\ if \ i=j                         
	\end{array}\right.
\end{equation*}
is the kronecker delta. Notice that if $a_{n-1}=\alpha(n-1)$ and $b_k=\alpha k$, the above recurrence relation becomes \eqref{recwl}. Hence, for $a_i=\alpha i$ and $b_j=\alpha j$, the formula \cite[Theorem 1.1]{Mansour}
\begin{equation}
u(n,k)=\sum_{j=0}^k\left(\frac{\prod_{i=0}^{n-1}(b_j+a_i)}{\prod_{i=0,i\neq j}^{n-1}(b_j-b_i)}\right)
\end{equation}
can be utilized to obtain an explicit formula for $\widehat{w}_{(\alpha)}(n,k)$ given in the next theorem.
\begin{theorem}
The translated Whitney-Lah numbers satisfy the following explicit formula:
\begin{equation}
\widehat{w}_{(\alpha)}(n,k)=\frac{\alpha^{n-k}}{k!}\sum_{j=0}^k(-1)^{k-j}\binom{k}{j}\left\langle j\right\rangle_n.\label{r1}
\end{equation}
\end{theorem}
This theorem enables us to write the numbers $\widehat{w}_{(\alpha)}(n,k)$ in a closed form similar to \eqref{lah1}. It is implied in the proof of the succeeding corollary.
\begin{corollary}
The translated Whitney-Lah numbers satisfy the following relation:
\begin{equation}
\widehat{w}_{(\alpha)}(n,k)=\alpha^{n-k}L(n,k).\label{r2}
\end{equation}
\end{corollary}
\begin{proof}
Since $\left\langle j\right\rangle_n=(j+n-1)_n$, then
\begin{eqnarray*}
\widehat{w}_{(\alpha)}(n,k)&=&\frac{\alpha^{n-k}}{k!}\sum_{j=0}^k(-1)^{k-j}\binom{k}{j}(j+n-1)_n\\
&=&\alpha^{n-k}\frac{n!}{k!}\sum_{j=0}^k(-1)^{k-j}\binom{k}{j}\binom{j+n-1}{n}.
\end{eqnarray*}
From \cite[Identity 5.24]{Graham}, it is known that the binomial coefficients satisfy the following useful identity:
\begin{equation}
\sum_{j}\binom{\ell}{m+j}\binom{s+j}{n}(-1)^j=(-1)^{\ell+m}\binom{s-m}{n-\ell}.\label{graham}
\end{equation}
Hence, with $m=0$, $\ell=k$ and $s=n-1$, we obtain
\begin{equation}
\widehat{w}_{(\alpha)}(n,k)=\alpha^{n-k}\frac{n!}{k!}\binom{n-1}{n-k}.\label{r2.1}
\end{equation}
This completes the proof.
\end{proof}
\begin{corollary}
The translated Whitney-Lah numbers satisfy the following exponential generating function:
\begin{equation}
\sum_{n=k}^{\infty}\widehat{w}_{(\alpha)}(n,k)\frac{t^n}{n!}=\frac{1}{k!}\left(\frac{t}{1-\alpha t}\right)^k.\label{r3}
\end{equation}
\end{corollary}
\begin{proof}
Applying \eqref{r1}, and the binomial and negative binomial expansions,
\begin{eqnarray*}
\sum_{n=k}^{\infty}\widehat{w}_{(\alpha)}(n,k)\frac{t^n}{n!}&=&\frac{1}{\alpha^kk!}\sum_{j=0}^k(-1)^{k-j}\binom{k}{j}\sum_{n=k}^{\infty}(\alpha t)^n\binom{j+n-1}{n}\\
&=&\frac{1}{\alpha^kk!}\sum_{j=0}^k(-1)^{k-j}\binom{k}{j}(1-\alpha t)^{-j}\\
&=&\frac{1}{\alpha^k}\left[(1-\alpha t)^{-1}-1\right]^k\\
&=&\frac{1}{k!}\left(\frac{t}{1-\alpha t}\right)^k.
\end{eqnarray*}
\end{proof}
Clearly, the results shown in the previous corollaries generalize identities \eqref{lah1} and \eqref{lah3} for the classical Lah numbers when $\alpha=1$. The binomial identity in \eqref{graham} can also be utilized to derive another interesting formula for the translated Whitney-Lah numbers. By setting $s=n$, $\ell=k-1$ and $m=-1$,
\begin{equation*}
\sum_{j=1}^k\binom{k-1}{j-1}\binom{n+j}{n}(-1)^j=(-1)^{k-2}\binom{n+1}{n-k+1}.
\end{equation*}
Multiplying both sides by $k!$ gives
\begin{equation*}
\sum_{j=1}^k\binom{k-1}{j-1}\binom{n+j}{n}(-1)^j=\sum_{j=1}^k\widehat{w}_{(\alpha)}(k,j)\frac{(n+j)!(-1)^j}{n!\alpha^{k-j}}
\end{equation*}
in the left-hand side after using \eqref{r2.1}. On the other hand, the right-hand side simply becomes
\begin{equation*}
(-1)^{k-2}\binom{n+1}{n-k+1}=(-1)^k\frac{(n+1)!}{(n-k+1)!}.
\end{equation*}
Thus, we have derived the following theorem:
\begin{theorem}\label{thm1}
For $k\geq2$ and $n\geq k-1$, the translated Whitney-Lah numbers satisfy
\begin{equation}
\sum_{j=1}^k(-\alpha)^j\widehat{w}_{(\alpha)}(k,j)(n+j)!=(-\alpha)^k\frac{n!(n+1)!}{(n-k+1)!}.\label{r4}
\end{equation}
\end{theorem}
When $\alpha=1$, we immediately recognize
\begin{equation}
\sum_{j=1}^k(-1)^jL(k,j)(n+j)!=(-1)^k\frac{n!(n+1)!}{(n-k+1)!},\label{gouqi}
\end{equation}
an identity for the classical Lah numbers which was proved using six different methods by Guo and Qi \cite{Guo}. A more direct approach in establishing \eqref{r4} is as follows.
\begin{proof}[Alternative proof of Theorem \ref{thm1}]
The generating function in \eqref{twlHGF} may be rewritten as
\begin{equation}
(-\alpha)^k(-t)_k=\sum_{j=0}^k\alpha^k\widehat{w}_{(\alpha)}(k,j)(t)_j.
\end{equation}
Since $(-n-1)_jn!=(-1)^j(n+j)!$, then replacing $t$ with $-n-1$ in the previous equation gives
\begin{equation*}
(-\alpha)^kn!(n+1)_k=\sum_{j=0}^k(-\alpha)^j\widehat{w}_{(\alpha)}(k,j)(n+j)!
\end{equation*}
as desired.
\end{proof}

Now, to derive a generalization of \eqref{QiF}, there are two known methods presented in the paper of Qi \cite{Qi} to choose from. The first one employs the Faa di Bruno's formula and the $n$-th derivative of the exponential function $e^{\pm 1/x}$ given by
\begin{equation*}
\left(e^{\pm 1/x}\right)^{(n)}=(-1)^ne^{\pm 1/x}\sum_{k=1}^n(\pm 1)^kL(n,k)\frac{1}{t^{n+k}}
\end{equation*}
found in the paper of Daboud et al. \cite{Dab}. The second is less complicated and requires only the use of the inverse relation
\begin{equation}
f_n=\sum_{j=0}^n\stirlingf{n}{j}g_j\Longleftrightarrow g_n=\sum_{j=0}^n(-1)^{n-j}\stirlings{n}{j}f_j.
\end{equation}
To obtain the next objective, we choose a process similar to the latter since by using the orthogonal relations \cite[Corollary 4.2]{Mangontarum1}
\begin{equation*}
\sum_{j=m}^n(-1)^{j-m}\widetilde{W}_{(\alpha)}(n,j)\widetilde{w}_{(\alpha)}(j,m)=\sum_{j=m}^n(-1)^{n-j}\widetilde{w}_{(\alpha)}(n,j)\widetilde{W}_{(\alpha)}(j,m)=\delta_{m,n},
\end{equation*}
it can be easily shown that the following inverse relation is valid:
\begin{equation}
f_n=\sum_{j=0}^n\widetilde{w}_{(\alpha)}(n,j)g_j\Longleftrightarrow g_n=\sum_{j=0}^n(-1)^{n-j}\widetilde{w}_{(\alpha)}(n,j)f_j.
\end{equation}
Next, we rewrite \eqref{wlahwW} as
\begin{equation}
\widehat{w}_{(\alpha)}(n,k)=\sum_{j=0}^n\widetilde{w}_{(\alpha)}(n,j)\widetilde{W}_{(\alpha)}(j,k),\label{wlahwW2}
\end{equation}
and take $g_j=\widetilde{W}_{(\alpha)}(j,k)$ and $f_n=\widehat{w}_{(\alpha)}(n,k)$ so that when the above inverse relation is applied, we get
\begin{equation}
\widetilde{W}_{(\alpha)}(n,k)=\sum_{j=0}^n(-1)^{n-j}\widetilde{W}_{(\alpha)}(n,j)\widehat{w}_{(\alpha)}(j,k).\label{wlahwW3}
\end{equation}
We then recall that the translated Dowling numbers \cite{Mangontarum2}, denoted by $D_{(\alpha)}(n)$, are defined as the sum of the translated Whitney numbers of the second kind, i.e.
\begin{equation}
D_{(\alpha)}(n)=\sum_{k=0}^n\widetilde{W}_{(\alpha)}(n,k),\label{translatedDN}
\end{equation}
and are known to satisfy the explicit formula \cite[Equation 26]{Mangontarum2}
\begin{equation*}
D_{(\alpha)}(n)=\left(\frac{1}{e}\right)^{1/\alpha}\sum_{i=0}^{\infty}\frac{(i\alpha)^n}{i!\alpha^i}.
\end{equation*}
Summing both sides of \eqref{wlahwW3} up to $n$ and appyling \eqref{translatedDN} gives
\begin{equation*}
D_{(\alpha)}(n)=\sum_{k=0}^n\sum_{j=0}^n(-1)^{n-j}\widetilde{W}_{(\alpha)}(n,j)\widehat{w}_{(\alpha)}(j,k).
\end{equation*}
Thus, we have proved the result in the next theorem.
\begin{theorem}
The translated Dowling numbers satisfy the explicit formula given by
\begin{equation}
D_{(\alpha)}(n)=\sum_{j=0}^n(-1)^{n-j}\left(\sum_{k=0}^j\widehat{w}_{(\alpha)}(j,k)\right)\widetilde{W}_{(\alpha)}(n,j).\label{GQiF1}
\end{equation}
\end{theorem}
To close this section, notice that by \eqref{r2}, we may write
\begin{equation*}
D_{(\alpha)}(n)=\sum_{j=0}^n(-1)^{n-j}\left(\sum_{k=0}^j\alpha^{j-k}L(j,k)\right)\widetilde{W}_{(\alpha)}(n,j).
\end{equation*}
Since it is known that \cite{Mangontarum1,Mangontarum2} $\widetilde{W}_{(1)}(n,j)=\stirlings{n}{j}$ and $D_{(1)}(n)=B_n$, it shows that the formula in \eqref{GQiF1} generalizes Qi's formula in \eqref{QiF} when $\alpha=1$.

\section{Translated $q$-Whitney-Lah numbers}\label{sec2}

The translated $q$-Whitney numbers of the first and second kind \cite{Mah4}, denoted by $w^1_{(\alpha)}[n,k]_q$ and $w^2_{(\alpha)}[n,k]_q$, respectively, are defined in terms of the following horizontal generating functions:
\begin{equation}
[t|\alpha]_n=\sum_{k=0}^nw^1_{(\alpha)}[n,k]_q[t]_q^k\label{def1}
\end{equation}
and
\begin{equation}
[t]_q^n=\sum_{k=0}^nw^2_{(\alpha)}[n,k]_q[t|\alpha]_k,\label{def2}
\end{equation}
where 
\begin{equation*}
[t|\alpha]_n=\prod_{i=0}^{n-1}[t-i\alpha]_q.
\end{equation*}
Here, $[n]_q$ is used to denote the $q$-analogue of the integer $n$ defined by 
\begin{equation*}
[n]_q=\frac{q^n-1}{q-1}=1+q+q^2+\cdots+q^{n-1}.
\end{equation*}
Various combinatorial properties of the numbers $w^1_{(\alpha)}[n,k]_q$ and $w^2_{(\alpha)}[n,k]_q$ and a certain combinatorial interpretation in the context of $A$-tableaux have already been established in the same paper. The properties include the inverse relation \cite[Corollary 2.10]{Mah4}
\begin{equation}
f_n=\sum_{j=0}^nw^1_{(\alpha)}[n,j]_qg_j\Longleftrightarrow g_n=\sum_{j=0}^nw^2_{(\alpha)}[n,j]_qf_j. \label{invqTW}
\end{equation}
Moreover, the said numbers have been shown to be proper $q$-analogues of the translated Whitney numbers. In general, the term ``$q$-analogue'' refers to a mathematical expression in terms of a parameter $q$ such that as $q\rightarrow 1$, it reduces to a known identity or formula. For instance, it is clear that 
\begin{equation*}
\lim_{q\rightarrow 1}[n]_q=n.
\end{equation*}
Other examples are the $q$-binomial coefficient 
\begin{equation*}
\binom{n}{k}_q=\prod_{j=1}^k\frac{q^{n-j+1}-1}{q^j-1}=\frac{[n]_q!}{[k]_q![n-k]_q!}
\end{equation*}
and the $q$-falling factorial of $n$ of order $k$
\begin{equation*}
[n]_{q,k}=\prod_{j=0}^{k-1}\frac{q^{n-j}-1}{q-1}=\frac{[n]_q!}{}[n-k]_q!,
\end{equation*}
where $[n]_q!=\prod_{i=1}^n[i]_q$ is the $q$-factorial of $n$. The following limits are easy to verify:
\begin{equation*}
\lim_{q\rightarrow 1}[n]_q!=n!,\ \ \lim_{q\rightarrow 1}\binom{n}{k}_q=\binom{n}{k},\ \ \lim_{q\rightarrow 1}[n]_{q,k}=(n)_k.
\end{equation*}
The book of Kac and Cheung \cite{Kac} is a rich source for further discussions on $q$-analogues. The study of $q$-analogues of mathematical identities has been the interest of many mathematicians over a long period of time. For the case of the Lah numbers, Garsia and Remmel \cite{Gar} defined the $q$-Lah numbers, denoted by $L_q(n,k)$, by
\begin{equation}
[t]_q[t+1]_q\cdots[t+n-1]_q=\sum_{k=0}^nL_q(n,k)[t]_q[t-1]_q\cdots[t-k+1]_q
\end{equation}
with the recurrence relation
\begin{equation}
L_q(n+1,k)=q^{n+k-1}L_q(n,k-1)+[n+k]_qL_q(n,k)
\end{equation}
and explicit formula
\begin{equation}
L_q(n,k)=\binom{n}{k}_q\frac{[n-1]_q!}{[k-1]_q!}q^{k(k-1)}.
\end{equation}
A more general notion was also introduced in \cite{Mah4}. The translated $q$-Whitney numbers of the third kind, denoted by $L_{(\alpha)}[n,k]_q$, are defined as coefficients in the expansion of \cite[Equation 15]{Mah4}
\begin{equation}
[t|-\alpha]_n=\sum_{k=0}^nL_{(\alpha)}[n,k]_q[t|\alpha]_k\label{def3}
\end{equation}
which can be computed recursively using the formula \cite[Equation 31]{Mah4}
\begin{equation}
L_{(\alpha)}[n+1,k]_q=q^{\alpha(n+k-1)}L_{(\alpha)}[n,k-1]_q+[\alpha(n+k)]_qL_{(\alpha)}[n,k]_q.
\end{equation}
It is easy to see that $L_{(1)}[n,k]_q=L_q(n,k)$.

\begin{theorem}
The numbers $L_{(\alpha)}[n,k]_q$ satisfy the following:
\begin{equation}
L_{(\alpha)}[n,k]_q=\sum_{j=0}^nw^1_{(-\alpha)}[n,j]_qw^2_{(\alpha)}[j,k]_q.\label{qw1w2}
\end{equation}
\end{theorem}
\begin{proof}
Putting $-\alpha$ in place of $\alpha$ in \eqref{def1} and by applying \eqref{def2},
\begin{eqnarray*}
[t|-\alpha]_n&=&\sum_{k=0}^nw^1_{(-\alpha)}[n,k]_q[t]_q^k\\
&=&\sum_{j=0}^n\left\{\sum_{k=j}^nw^1_{(-\alpha)}[n,k]_qw^2_{(\alpha)}[k,j]_q\right\}[t|\alpha]_j.
\end{eqnarray*}
By comparing the coefficients of $[t|\alpha]_j$ in the last equation with that of \eqref{def3}, we get the desired result.
\end{proof}
The identity in the previous theorem suggests that the numbers $L_{(\alpha)}[n,k]_q$ may be referred to as the translated $q$-Whitney-Lah numbers. To establish an explicit formula, we will use a method different from the one used in the previous section. We start by rewriting \eqref{def3} into the form
\begin{eqnarray*}
[\alpha k|-\alpha]_n&=&\sum_{j=0}^nL_{(\alpha)}[n,j]_q[\alpha k|\alpha]_j\\
&=&\sum_{j=0}^k\binom{k}{j}_{q^{\alpha}}\left\{\frac{L_{(\alpha)}[n,j]_q[\alpha k|\alpha]_j}{\binom{k}{j}_{q^{\alpha}}}\right\}.
\end{eqnarray*}
Since the well-known $q$-binomial inversion formula can be expressed as
\begin{equation}
f_k=\sum_{j=0}^k\binom{k}{j}_{q^{\alpha}}g_j\Longleftrightarrow g_k=\sum_{j=0}^k(-1)^{k-j}q^{\alpha\binom{k-j}{2}}\binom{k}{j}_{q^{\alpha}}f_j,
\end{equation}
then with $f_k=[\alpha k|-\alpha]_q$ and $g_j=\frac{L_{(\alpha)}[n,j]_q[\alpha k|\alpha]_j}{\binom{k}{j}_{q^{\alpha}}}$, we get
\begin{equation*}
[\alpha k|\alpha]_kL_{(\alpha)}[n,k]_q=\sum_{j=0}^k(-1)^{k-j}q^{\alpha\binom{k-j}{2}}\binom{k}{j}_{q^{\alpha}}[\alpha j|-\alpha]_n,
\end{equation*}
the result in the next theorem.
\begin{theorem}
The translated $q$-Whitney-Lah numbers satisfy the following explicit formula:
\begin{equation}
L_{(\alpha)}[n,k]_q=\frac{1}{[k]_{q^{\alpha}}![\alpha]_q^k}\sum_{j=0}^k(-1)^{k-j}q^{\alpha\binom{k-j}{2}}\binom{k}{j}_{q^{\alpha}}[\alpha j|-\alpha]_n.\label{qr1}
\end{equation}
\end{theorem}
Formula \eqref{qr1} is a $q$-analogue of the explicit formula in \eqref{r1} since
\begin{equation*}
\lim_{q\rightarrow1}[k]_{q^{\alpha}}!=k!,\ \ \lim_{q\rightarrow1}[\alpha j|\alpha]_n=\alpha^n\left\langle j\right\rangle_n
\end{equation*}
and
\begin{eqnarray*}
\lim_{q\rightarrow1}L_{(\alpha)}[n,k]_q&=&\lim_{q\rightarrow1}\left(\frac{1}{[k]_{q^{\alpha}}![\alpha]_q^k}\sum_{j=0}^k(-1)^{k-j}q^{\alpha\binom{k-j}{2}}\binom{k}{j}_{q^{\alpha}}[\alpha j|-\alpha]_n\right)\\
&=&\frac{\alpha^{n-k}}{k!}\sum_{j=0}^k(-1)^{k-j}\binom{k}{j}\left\langle j\right\rangle_n.
\end{eqnarray*}
Furthermore, we may use the above explicit formula in establishing a kind of exponential generating function for the numbers $L_{(\alpha)}[n,k]_q$. But before proceeding, we first mention the following important identities:
\begin{equation}
[\alpha j|-\alpha]_n=[\alpha]_q^n[j+n-1]_{q^{\alpha},n},\ \ \frac{[j+n-1]_{q^{\alpha},n}}{[n]_{q^{\alpha}}!}=\binom{j+n-1}{n}_{q^{\alpha}}\label{pe1}
\end{equation}
and
\begin{equation}
\prod_{k=0}^{n-1}\frac{1}{1-q^kt}=\sum_{k=0}^{\infty}\binom{n+k-1}{k}_qt^k.\label{pe2}
\end{equation}
\begin{corollary}
The translated $q$-Whitney-Lah numbers satisfy the following exponential generating function:
\begin{equation}
\sum_{n=0}^{\infty}L_{(\alpha)}[n,k]_q\frac{t^n}{[n]_{q^{\alpha}}!}=\frac{1}{[k]_{q^{\alpha}}![\alpha]_q^k}\sum_{j=0}^k(-1)^{k-j}q^{\alpha\binom{k-j}{2}}\binom{k}{j}_{q^{\alpha}}\prod_{n=0}^{j-1}(1-q^{\alpha n}[\alpha]_qt)^{-1}.\label{qr1.1}
\end{equation}
\end{corollary}
\begin{proof}
From \eqref{qr1} and \eqref{pe1},
\begin{equation*}
\sum_{n=0}^{\infty}L_{(\alpha)}[n,k]_q\frac{t^n}{[n]_{q^{\alpha}}!}=\frac{1}{[k]_{q^{\alpha}}![\alpha]_q^k}\sum_{j=0}^k(-1)^{k-j}q^{\alpha\binom{k-j}{2}}\binom{k}{j}_{q^{\alpha}}\sum_{n=0}^{\infty}\binom{j+n-1}{n}_{q^{\alpha}}([\alpha]_qt)^n.
\end{equation*}
The result is obtained by applying \eqref{pe2} in the second summation.
\end{proof}
By taking the limit of \eqref{qr1.1} as $q\rightarrow1$,
\begin{equation*}
\lim_{q\rightarrow1}\sum_{n=0}^{\infty}L_{(\alpha)}[n,k]_q\frac{t^n}{[n]_{q^{\alpha}}!}=\frac{1}{\alpha^kk!}\sum_{j=0}^k(-1)^{k-j}\binom{k}{j}\left(\frac{1}{1-\alpha t}\right)^j
\end{equation*}
which in turn simplifies to \eqref{r3}. On the other hand, the next theorem contains a $q$-analogue of \eqref{r4}.
\begin{theorem}
The translated $q$-Whitney-Lah numbers satisfy the following:
\begin{equation}
\sum_{j=0}^k(-[\alpha]_q)^jq^{-nj-\binom{j+1}{2}}L_{(\alpha)}[k,j]_q[n+j]_{q^{\alpha}}!=\frac{(-[\alpha]_q)^k[n]_{q^{\alpha}}![n+1]_{q^{\alpha}}!}{[n-k+1]_{q^{\alpha}}!}.\label{qr2}
\end{equation}
\end{theorem}
\begin{proof}
The proof is somewhat parallel to the alternative proof of Theorem \ref{thm1}. We proceed by rewriting \eqref{def3} as
\begin{equation}
[-\alpha]_q^k\prod_{i=0}^{k-1}[-t-i]_{q^{\alpha}}=\sum_{j=0}^k[\alpha]_q^jL_{(\alpha)}[k,j]_q\prod_{i=0}^{j-1}[t-i]_{q^{\alpha}}.
\end{equation}
We put $-n-1$ in place of $t$ and multiply both sides by $[n]_{q^{\alpha}}!$ so that the left-hand side becomes
\begin{eqnarray*}
[-\alpha]_q^k\prod_{i=0}^{k-1}[n+1-i]_{q^{\alpha}}[n]_{q^{\alpha}}!&=&[-\alpha]_q^k[n]_{q^{\alpha}}![n+1]_{q^{\alpha},k}\\
&=&\frac{[-\alpha]_q^k[n]_{q^{\alpha}}![n+1]_{q^{\alpha}}!}{[n-k+1]_{q^{\alpha}}!}
\end{eqnarray*}
while the right-hand side is
\begin{equation*}
\sum_{j=0}^k[\alpha]_q^jL_{(\alpha)}[k,j]_q[n]_{q^{\alpha}}!\prod_{i=0}^{j-1}[t-i]_{q^{\alpha}}=\sum_{j=0}^k(-[\alpha]_q)^jq^{-nj-\binom{j+1}{2}}L_{(\alpha)}[k,j]_q[n+j]_{q^{\alpha}}!,
\end{equation*}
where the identity $j(n+1)+\binom{j}{2}=nj+\binom{j+1}{2}$ is used. Combining these equations give the desired result.
\end{proof}

The corollary below is a direct consequence of \eqref{qr2} when $\alpha=1$. This formula is a $q$-analogue of Guo and Qi's \cite{Guo} identity in \eqref{gouqi} which can easily be verified by taking the limit as $q\rightarrow1$.
\begin{corollary}
The $q$-Lah numbers satisfy
\begin{equation}
\sum_{j=0}^k(-1)^jq^{-nj-\binom{j+1}{2}}L_q(k,j)[n+j]_q!=\frac{(-1)^k[n]_q![n+1]_q}{[n-k+1]_q}.\label{qr2.1}
\end{equation}
\end{corollary}

The translated $q$-Dowling numbers \cite{Mah4}, denoted by $D_{(\alpha)}[n]_q$, are defined by the following sum:
\begin{equation}
D_{(\alpha)}[n]_q=\sum_{k=0}^nw^2_{(\alpha)}[n,k]_q.
\end{equation}
The last theorem presents a $q$-analogue of the explicit formula in \eqref{GQiF1}.
\begin{theorem}
The translated $q$-Dowling numbers satisfy the following explicit formula
\begin{equation}
D_{(\alpha)}[n]_q=\sum_{j=0}^n\left(\sum_{j=0}^kL_{(\alpha)}[j,k]_q\right)w_{(-\alpha)}^2[n,j]_q.\label{qGQiF1}
\end{equation}
\end{theorem} 
\begin{proof}
We put $-\alpha$ in place of $\alpha$, and set $g_j=w^2_{(\alpha)}[j,k]_q$ and $f_n=L_{(\alpha)}[n,k]_q$ in the inverse relation in \eqref{invqTW} so that when the resulting relation is applied to \eqref{qw1w2},
\begin{equation*}
w^2_{(\alpha)}[n,k]_q=\sum_{j=0}^nw^2_{(-\alpha)}[n,j]_qL_{(\alpha)}[j,k]_q.
\end{equation*}
The desired result is obtained by summing over up to $n$.
\end{proof}
The explicit formula \cite[Equation 10]{Mangontarum2}
\begin{equation*}
\widetilde{W}_{(\alpha)}(n,k)=\frac{1}{\alpha^kk!}\sum_{j=0}^k(-1)^{k-j}\binom{k}{j}(\alpha j)^n
\end{equation*}
shows that $\widetilde{W}_{(-\alpha)}(n,k)=(-1)^{n-k}\widetilde{W}_{(\alpha)}(n,k)$. Hence,
\begin{eqnarray*}
\lim_{q\rightarrow1}D_{(\alpha)}[n]_q&=&\lim_{q\rightarrow1}\sum_{j=0}^n\left(\sum_{j=0}^kL_{(\alpha)}[j,k]_q\right)w_{(-\alpha)}^2[n,j]_q\\
&=&\sum_{j=0}^n(-1)^{n-j}\left(\sum_{k=0}^j\widehat{w}_{(\alpha)}(j,k)\right)\widetilde{W}_{(\alpha)}(n,j)
\end{eqnarray*}
which is precisely \eqref{GQiF1}.

As we end, it may be worthwhile to say that the present paper was not able to express the explicit formula of $L_{(\alpha)}[n,k]_q$ in a way similar to that of \eqref{r2.1} for the case of $\widehat{w}_{(\alpha)}(n,k)$. Perhaps this can be done by establishing a $q$-analogue of the binomial identity in \eqref{graham}.


\begin{thebibliography}{99}

\bibitem{Bel}
{H. Belbachir and I. Bousbaa, Translated Whitney and $r$-Whitney numbers: a combinatorial approach, {\it J. Integer Seq.} {\bf 16} (2013), Article 13.8.6.}

\bibitem{Chen}
{C. Chen and K. Kho, {\it Principles and Techniques in Combinatorics}, World Scientific Publishing Co., 1992.}

\bibitem{Comt}
{L. Comtet, {\it Advanced Combinatorics}, D. Reidel Publishing Co., 1974.}

\bibitem{Dab}
{S. Daboul, J. Mangaldan, M. Z. Spivey, and P. J. Taylor, The Lah numbers and the $n$th derivative of  $e^{1/x}$, {\it Math. Magazine} {\bf 86} (2013), 39--47.}

\bibitem{Gar}
{A. M. Garsia and J. Remmel, A combinatorial interpretation of $q$-derangement and $q$-Laguerre numbers, {\it Europ. J. Combinatorics} {\bf 1} (1980), 47--59.}

\bibitem{Graham}
{R. L. Graham, D. E. Knuth, and O. Patashnik, {\it Concrete Mathematics: A Foundation for Computer Science}, Addison-Wesley, 1994.}

\bibitem{Guo}
{B. Guo and F. Qi, Six proofs for an identity of the Lah numbers, {\it Online J. Anal. Comb.} {\bf 10} (2015).}

\bibitem{Kac}
{V. Kac and P. Cheung, {\it Quantum Calculus}, Springer, New York, NY, USA, (2002).}

\bibitem{Mangontarum1}
{M. M. Mangontarum and A. M. Dibagulun, On the translated Whitney numbers and their combinatorial properties, {\it British J. Appl. Sci. Technology} {\bf 11} (2015), 1--15.}

\bibitem{Mah4}
M. M. Mangontarum, O. I. Cauntongan, and A. M. Dibagulun, A note on the translated Whitney numbers and their $q$-analogues, {\it Turkish J. of Analysis and Number Theory} {\bf 4} (2016), 74--81.

\bibitem{Mangontarum2}
{M. M. Mangontarum, A. P. M.-Ringia, and N. S. Abdulcarim, The translated Dowling polynomials and numbers, {\it International Scholarly Research Notices} {\bf 2014}, Article ID 678408, 8 pages, (2014).}

\bibitem{Mansour}
{T. Mansour, S. Mulay, and M. Shattuck, A general two-term recurrence and its solution, {\it European J. Combin.} {\bf 33} (2012) 20--26.}

\bibitem{Pet}
{M. Petkov\v{s}ek and T. Pisanski, Combinatorial interpretation of unsigned Stirling and Lah numbers, {\it Pi Mu Epsilon J.} {\bf 12} (7) (2007), 417--424.}

\bibitem{Qi}
{F. Qi, An explicit formula for the Bell numbers in terms of Lah and Stirling numbers, {\it Mediterr. J. Math.} {\bf 13} (2016), 2795--2800.}

\end{thebibliography}
\end{document}